\documentclass[reqno,11pt]{amsart}
\usepackage[all]{xy}
\usepackage{amsthm, amssymb, amsmath, amscd}
\usepackage{braket}
\usepackage{geometry}
\usepackage{mathrsfs}
\usepackage{threeparttable, mathtools, multirow, makecell}
\def\rad{\operatorname {rad}}

\usepackage[dvips]{graphicx, color}

\def\a{\alpha}
\def\b{\beta}
\def\c{\gamma}
\def\d{\delta}

\def\l{\lambda}

\def\NN{{\mathbb N}}
\def\PP{{\mathbb P}}

\def\ZZ{{\mathbb Z}}

\def\cal{\mathcal}

\def\cA{{\cal A}}
\def\cB{{\cal B}}
\def\cC{{\cal C}}
\def\cD{{\cal D}}

\def\cF{{\cal F}}
\def\cG{{\cal G}}

\def\Aut{\operatorname{Aut}}

\def\depth{\operatorname{depth}}

\def\det{\operatorname{det}}
\def\dim{\operatorname{dim}}

\def\Ext{\operatorname {Ext}}

\def\GL{\operatorname {GL}}

\def\gcd{\operatorname{gcd}}

\def\GL{\operatorname{GL}}
\def\gldim{\operatorname{gldim}}

\def\grmod{\operatorname{grmod}}
\def\GrMod{\operatorname{GrMod}}

\def\id{\operatorname {id}}

\def\Kdim{\operatorname{Kdim}}

\def\mod{\operatorname{mod}}
\def\Mod{\operatorname{Mod}}

\def\Proj{\operatorname{Proj}}

\def\Spec{\operatorname {Spec}}

\def\uExt{\operatorname{\underline{Ext}}}

\def\uExt{\operatorname{\underline{Ext}}}

\def\uCM{\underline {\operatorname{CM}}}

\def\<{\langle}
\def\>{\rangle}

\def\sD{\mathscr D}

\def\sH{\mathscr H}

\usepackage[dvips]{graphicx, color}

\def\rnum#1{\expandafter{\romannumeral #1}}
\def\Rnum#1{\uppercase\expandafter{\romannumeral #1}}

\theoremstyle{plain} 
\newtheorem{theorem}{Theorem}[section]
\newtheorem{corollary}[theorem]{Corollary}
\newtheorem{lemma}[theorem]{Lemma}
\newtheorem{proposition}[theorem]{Proposition}

\theoremstyle{definition}
\newtheorem{definition}[theorem]{Definition}
\newtheorem{example}[theorem]{Example}

\theoremstyle{remark}

\newtheorem{remark}[theorem]{Remark}

\numberwithin{equation}{section}

\begin{document}

\title{Noncommutative affine pencils of conics}

\author{Haigang Hu, Izuru Mori, Koki Takeda, Wenchao Wu}

\address{${^1}$ School of Mathematical Sciences, University of Science and Technology of China, Hefei Anhui 230026, CHINA}

\email{huhaigang@ustc.edu.cn \textnormal{(H. Hu)}}

\address{${^2}$ Department of Mathematics, Faculty of Science, Shizuoka University, Shizuoka 422-8529, JAPAN}

\email{mori.izuru@shizuoka.ac.jp \textnormal{(I. Mori)}}

\address{${^3}$ Graduate School of Science and Technology, Shizuoka University, Shizuoka 422-8529, JAPAN}

\email{takeda.koki.23@shizuoka.ac.jp \textnormal{(K. Takeda)}}

\address{${^4}$ School of Mathematical Sciences, University of Science and Technology of China, Hefei Anhui 230026, CHINA}

\email{wuwch20@mail.ustc.edu.cn \textnormal{(W. Wu)}}


\thanks{The first author was supported by the National Natural Science Foundation of China (No.\,12371042)
and the Youth Innovation Fund of USTC (No.\,WK0010000086). The second author was supported by JSPS Grant-in-Aid for Scientific Research (C)  Grant Number JP25K06917.}

\begin{abstract} 
This paper is one of the series of papers which are dedicated to the complete classification of noncommutative conics. 
In this paper, we define and study noncommutative affine pencils of conics, and give a complete classification result.
We also fully classify $4$-dimensional Frobenius algebras. It turns out that the classification of noncommutative affine pencils of conics is the same as the classification of $4$-dimensional Frobenius algebras. 
\end{abstract}

\maketitle

\section{Introduction}  

Throughout this paper, let $k$ be an algebraically closed field of characteristic $0$. 

In noncommutative algebraic geometry, an $n$-dimensional quantum polynomial algebra $S$ defined below is a noncommutative analogue of the commutative polynomial algebra $k[x_1, \dots, x_n]$.

\begin{definition}[\cite{AS}]
A noetherian connected graded algebra $S$ is called an {\it $n$-dimensional quantum polynomial  algebra} if 
\begin{enumerate}
\item{} $\gldim S=n<\infty$,  
\item{} $H_S(t)=(1-t)^{-n}$, and   
\item{} (Gorenstein condition) $\uExt^i_S(k, S)\cong \begin{cases} 0 & \text{ if } i\neq n \\ k(n) & \text { if }  i=n. \end{cases}$
\end{enumerate}
\end{definition}

The {\it noncommutative projective scheme} $\operatorname{Proj_{nc}}S$ associated to $S$ in the sense of \cite{AZ} is regarded as a quantum $\PP^{n-1}$,  
so it is reasonable to define a noncommutative quadric hypersurface in a quantum $\PP^{n-1}$ as follows.

\begin{definition} We say that $A=S/(f)$ is (the homogeneous coordinate ring of) a {\it noncommutative quadric hypersurface} in a quantum $\PP^{n-1}$ if $S$ is an $n$-dimensional quantum polynomial algebra and $f\in S_2$ is a regular normal element.   In particular, if $n = 3$, then $A$ is called (the homogeneous coordinate ring of) a {\it noncommutative conic} (in a quantum $\PP^2$). 
\end{definition} 

The study of noncommutative quadric hypersurfaces is active in noncommutative algebraic geometry.  
Recently, there are many interesting and important results (\cite{H}, \cite{HMM}, \cite{HM}, \cite{MU}, \cite{SV}, etc). In particular, \cite{HMM} shows that there are exactly 9 isomorphism classes of noncommutative conics of the form $A=S/(f)$ where $S$ is a 3-dimensional ``Calabi-Yau'' quantum polynomial algebra and $f\in Z(S)_2$ is a homogeneous (regular) central element of degree 2  (see also \cite{HM}).  This is a surprising result since there are infinitely many isomorphism classes of 3-dimensional ``Calabi-Yau'' quantum polynomial algebras, and gives some hope to classify all noncommutative conics, dropping the condition ``Calabi-Yau''. 

This paper is one of the series of  papers (\cite{H}, \cite{HMM}, \cite{HM}) which are dedicated to the complete classification of noncommutative conics.
In this paper, we focus on defining and studying noncommutative affine pencils of conics (Section \ref{sec-napc}).
Let $F=(f_1, \dots, f_m)$ be a sequence of elements in a graded algebra $S$ where $\deg f_i = d_i$. We say that $F$ is a {\it strongly regular normal sequence} if both $F$ and $F^\vee := ((f_1)_{d_1}, \dots (f_m)_{d_m})$ are regular normal sequences.  

\begin{definition}
We say that $E = S/(f,g)$ is  (the coordinate ring of) a {\it noncommutative affine pencil of conics} if $S$ is a $2$-dimensional quantum polynomial algebra and $(f,g)$ is a strongly regular normal sequence of degree $2$ in $S$. 
\end{definition} 

In order to classify noncommutative affine pencils of conics, 
we need to find all strongly regular normal sequences $(f,g)$ of degree $2$ in each $2$-dimensional quantum polynomial algebra $S$. It is a difficult task. In Section \ref{sec-napc}, we provide some useful lemmas to reduce the calculation of finding regular normal elements of degree $2$ in graded algebras. For the purpose of the subsequent paper \cite{HMW}, we introduce the notion of $st$-equivalence and classify noncommutative affine pencils of conics up to $st$-equivalence (Theorem \ref{thm.ANC}). 

The study of Frobenius algebras is always an important subject. For a noncommutative conic $A$, Smith and Van den Bergh proved that the Cohen-Macaulay representation $\uCM^{\ZZ} A$ is equivalent to a bounded derived category $\cD^b(\mod C(A))$ for some $4$-dimensional Frobenius algebra $C(A)$ (\cite{SV}). Due to the importance, we give a complete classification of $4$-dimensional Frobenius algebras (Theorem \ref{thm-4dimfro}) in Section 4. 

By these classifications and the (de)homogenization theory,  we obtain our main result in Section \ref{sec-main}.

\begin{theorem}[Theorem \ref{thm.CF}]
The set of isomorphism classes of noncommutative affine pencils of conics is equal to the set of isomorphism classes of $4$-dimensional Frobenius algebras. 
\end{theorem}

We expect that there is a close relationship between noncommutative conics and 4-dimensional Frobenius algebras.  We have already mentioned that there is a way to construct a 4-dimensional Frobenius algebra $C(A)$ from a noncommutative conic $A$, however, we have no idea on how to construct a noncommutative conic from a 4-dimensional Frobenius algebra in general.  In the subsequent paper \cite{HMW}, we will discuss a way to construct a noncommutative conic from a noncommutative affine pencil of conics, which solves the problem by the above theorem. \\

\noindent {\bf Acknowledgement.} We would like to thank Ryoma Suzuki for his assistance in checking some of the computations. 

\section{Preliminaries} 

All algebras and vector spaces are over $k$. A graded algebra is a $\mathbb{Z}$-graded algebra.

For an algebra $R$, we denote by $\Mod R$ the category of right $R$-modules, and by $\mod R$ the full subcategory consisting of finitely generated modules. 

Let $A$ be a graded algebra. We denote by $\GrMod A$ the category of graded right $A$-modules with degree preserving right $A$-module homomorphisms, and by $\grmod A$ the full subcategory of $\GrMod A$ consisting of finitely generated graded right $A$-modules. For $M\in \GrMod A$ and $j \in \ZZ$, we define the shift  $M(j) \in \GrMod A$ by $M(j)_i := M_{j+i}$. For $M,N\in \GrMod A$, we define
$$
 \underline{\Ext}_A^i(M,N) := \bigoplus_{j\in \ZZ}{\Ext}^i_{\GrMod A}(M,N(j)). 
$$ 

For a graded algebra $A = \bigoplus_{i \in \mathbb{Z}} A_i$, we say that $A$ is {\it connected graded} if $A_0 = k$ and $A_i = 0$ for all $i < 0$, and $A$ is {\it locally finite} if $\dim_k A_i < \infty$ for all $i \in \mathbb{Z}$. If $A$ is a locally finite graded algebra, then we define the {\it Hilbert series} of $A$ by 
$
H_A(t) : = \sum_{i \in \mathbb{Z}} (\dim_{k} A_i) t^i \in \mathbb{Z}[[t,t^{-1}]].
$

Let $V$ be a finite dimensional vector space, and $T(V)$ the tensor algebra. Let $A = T(V)/(W)$ be a quadratic algebra, i.e., $W \subset V \otimes V$. The {\it quadratic dual} is defined by $A^! := T(V^*)/(W^\perp)$ where $V^*$ is the dual space of $V$, and $W^\perp$ is the orthogonal complement of $W$. 

\begin{definition}
A locally finite connected graded algebra $A$ is called a {\it Koszul algebra} if the trivial module $k_A:=A/A_{\geq 1}$ has a free resolution
$$
\xymatrix{
\cdots \ar[r] & P^i \ar[r] & \cdots \ar[r] & P^1 \ar[r] & P^0 \ar[r] & k_A \ar[r] & 0
}
$$
where $P^i$ is a graded free right $A$-module generated in degree $i$ for each $i \geq 0$. 
\end{definition}

Every Koszul algebra is quadratic. 

\begin{theorem} [{\cite[Theorem 5.9]{Sm96}}] \label{thm.KoHi}  If $A$ is a Koszul algebra, then $H_{A^!}(t)=1/H_A(-t)$.  
\end{theorem} 

\begin{theorem} [{\cite[Theorem 5.11]{Sm96}}] \label{thm.qpak} 
Every $n$-dimensional quantum polynomial algebra $S$ is Koszul so that $H_{S^!}(t)=(1+t)^n$.
\end{theorem}   

\begin{example} \label{exm-2dimqpa2}
It is one of the basic facts in noncommutative algebraic geometry that  every graded algebra $A=k\<x, y\>/(h)$ where $0\neq h\in k\<x, y\>_2$ is isomorphic as a graded algebra to exactly one of the following graded  algebras in Table \ref{tab-qaxyh} where $k_{\l}[x, y]\cong k_{\l'}[x, y]$ as graded algebras if and only if $\l'=\l^{\pm 1}$.  It follows that a graded algebra $S$ is a $2$-dimensional quantum polynomial algebra if and only if $S\cong k_J[x, y]$ or $S\cong k_\l[x, y]$ as graded algebras for some $\l\neq 0$.  

\begin{center} 
\begin{table}[h]
\caption{Quadratic algebras $A=k\<x, y\>/(h)$.} \label{tab-qaxyh}
\begin{tabular}{|c|c|c|c|c|c|}
\hline
$A$                & (noeth.) & $\gldim A$   & $H_A(t)$ & (Gor.) & (Kos.)                \\
\hline
$k\<x, y\>/(x^2)$ & No & $\infty$ & $(1+t)/(1-t-t^2)$ & No & Yes \\
\hline 
$k\<x, y\>/(xy)$ & No & 2 & $1/(1-t)^2$ & No & Yes \\
\hline 
 $k_J[x, y]:=k\<x, y\>/(xy-yx+y^2)$ & Yes   & 2 & $1/(1-t)^2$ & Yes & Yes                    \\ 
\hline 
$k_{\l}[x, y]:=k\<x, y\>/(xy-\l yx) \text{ where } \l\neq 0$   & Yes  & 2 & $1/(1-t)^2$ & Yes & Yes  \\   
\hline
\end{tabular} 
\end{table}
\end{center} 
\end{example}

\begin{definition} \label{def.gcliffod}
For a sequence  $M = (M_1, \dots, M_n)$ of linearly independent symmetric matrices of size $n$,  the {\it graded Clifford algebra} associated to $M$ is a graded algebra defined by
\begin{align*}
A: = k\<x_1, \dots, x_n, y_1, \dots, y_n\>/  (x_ix_j+x_jx_i-\sum_{m=1}^n(M_m)_{ij}y_m, & \\
x_iy_j-y_jx_i, y_iy_j-y_j & y_i)_{1\leq i, j\leq n}
\end{align*}
where $\deg x_i=1, \deg y_j=2$.   
\end{definition}

\begin{definition} Let $R$ be an algebra and $f\in R$.
\begin{enumerate}
\item{} We say that $f$ is right (left) regular if, for every $g\in R$,  $gf=0$ ($fg=0$) implies $g=0$. 
We say that $f$ is regular if it is both right and left regular. 
\item{} We say that $f$ is normal if $Rf=fR$.
\end{enumerate} 
\end{definition} 

We often implicitly assume that $f\in R$ is not a unit when we say that $f$ is a regular normal element.  In particular, for a connected graded algebra $A$, we often implicitly assume that $f\in A_{\geq 1}$ when we say that $f$ is a homogeneous regular normal element.

\begin{definition}  \label{defn-regseq}
Let $R$ be an algebra, and $F = (f_1, \dots, f_m)$ a sequence of elements in $R$.
\begin{enumerate}
\item[(1)] We say that  $F$ is {\it regular}  if $\bar f_i\in R/(f_1, \dots, f_{i-1})$ is regular for every $i=1, \dots, m$. 
\item[(2)] We say that $F$ is {\it normal (resp. central)} if $\bar f_i\in R/(f_1, \dots, f_{i-1})$ is normal (resp. central) for every $i=1, \dots, m$.  
\end{enumerate}
In case $R$ is a graded algebra: 
\begin{enumerate}
\item[(3)] We say that $F$ is of degree $d$
if $\deg f_i = d$  for every $i=1, \dots, m$.
\item[(4)] We say that $F$ is {\it homogeneous} if $f_i$ is homogeneous for every $i=1, \dots, m$.
\end{enumerate} 
\end{definition}

For $f_1, \dots, f_m\in R$, the notation $(f_1, \dots, f_m)$ has two possible meanings.  If we write $R/(f_1, \dots, f_m)$, then it is a two-sided ideal of $R$ generated by $f_1, \dots, f_m$.  If we write $F=(f_1, \dots, f_m)$, then it is a sequence of elements.  To avoid such a confusion, 
we often denote by $I_F$ the two-sided ideal of $R$ generated by  $f_1, \dots, f_m$ when $F=(f_1, \dots, f_m)$. 

\begin{definition} 
We define
\begin{align*}
\cA_{n, m} & : = \{S/(h_1, \dots, h_m) \mid \text{ $S$ is an $n$-dimensional quantum polynomial algebra and } \\
& \qquad \text  {$H=(h_1, \dots, h_m)$ is a homogeneous regular normal sequence of degree 2 in $S$}\}/\cong, \\
\cB_{n, m} & :=\{S/(h_1, \dots, h_m)\in \cA_{n, m}\mid S=k[x_1, \dots, x_n]\}/\cong, \\ 
\cC_{n, m} & :=\{S/(h_1, \dots, h_m)\in \cA_{n, m}\mid S \text { is a graded Clifford algebra and } h_1, \dots, h_m\in Z(S)_2\}/\cong. 
\end{align*}

A graded algebra $A$ is called a {\it noncommutative quadratic complete intersection} if $A\in \cA_{n, m}$.  In particular,  $A$ is called a {\it noncommutative conic}, a {\it noncommutative pencil of conics}, 
if $A\in \cA_{3, 1}, A\in \cA_{3,2}$, respectively. A commutative graded algebra $B$ is called a conic, a pencil of conics, 
if $B\in \cB_{3, 1}, B\in \cB_{3,2}$, respectively. 
\end{definition} 

We remark that $\cA_{n,0}$ is the set of isomorphism classes of $n$-dimensional quantum polynomial algebras,  and $\cB_{n,0} = \{k[x_1, \dots, x_n] \}$.

\begin{lemma}[{\cite[Corollary 1.4]{ST}, \cite[Section 5]{HY}}] \label{lem.asf} 
For every $A=S/(f)\in \cA_{n, 1}$, there exists a regular normal element $f^!\in A^!_2$ unique up to scalar such that $S^!=A^!/(f^!)$.  Moreover, if $f \in Z(S)_2$, then $f^!\in Z(A^!)_2$. 
\end{lemma} 

\begin{definition} For $A\in \cA_{n, 1}$, we define $C(A):=A^![{f^!}^{-1}]_0$. 
\end{definition} 

It is not obvious from the definition but $C(A)$ is independent of a choice of the pair $(S, f)$ such that $A=S/(f)\in \cA_{n, 1}$ up to isomorphism (see \cite[Lemma 5.3 (2)]{HMM}).  

We denote by $\cF_2$ the set of isomorphism classes of 4-dimensional Frobenius algebras, and by $\cG_2$ the set of isomorphism classes of 4-dimensional commutative Frobenius algebras.  In our earlier papers, we proved the following two important bijections.  

\begin{theorem} [{\cite[Theorem 3.28]{HM}}] \label{thm.HMd}  
The map $\cC_{n, m}\to \cB_{n, n-m}; \; A\mapsto A^!$ is a bijection for every $n\in \NN$ and $m=0, \dots, n$. 
\end{theorem} 

\begin{theorem} [{\cite[Corollary 4.10]{H}, \cite[Remark 3.29]{HM}}] \label{thm.CGb} 
The map $\cC_{3, 1}\to \cG_2; \; A \mapsto C(A)$ is a bijection. 
\end{theorem} 

\section{Noncommutative affine pencils of conics} \label{sec-napc}

In this section, we define and classify noncommutative affine pencils of conics.

\subsection{$st$-equivalences}

For the purpose of the subsequent paper \cite{HMW}, we classify noncommutative affine pencils of conics up to $st$-equivalence defined below.

\begin{definition} 
Let $S$ be a graded algebra, and $F = (f_1, \dots, f_m), F' = (f_1', \dots, f_m') $ sequences in $S$.
\begin{enumerate}
\item{} We write $F \sim _s F'$ if there exists $(\a_{ij})\in \GL_m(k)$ such that $f'_j=\sum _{i=1}^m\a_{ij}f_i$ for every $i=1, \dots, m$.  In the matrix notation, $(f_1', \dots, f_m')=(f_1, \dots, f_m)(\a_{ij})$.  
\item{} We write $F \sim _t F'$ if there exists $\phi\in \Aut^{\ZZ}S$  such that $f'_i=\phi(f_i)$ for every $i=1, \dots, m$.   
\item{} We write $F \sim _{st} F'$ if there exists a sequence $F''$ in $S$ such that  $F \sim_t F'' \sim _s F'.$ 
\end{enumerate}
\end{definition} 

\begin{remark} \label{rem.stc} 
It is clear that $F\sim_{st} F'$ implies $S/I_F\cong S/I_{F'}$, but the converse does not hold in general (see Example \ref{ex.stc}).   
\end{remark} 

\begin{remark} \label{rem.isst} 
Let $S=k[x_1, \dots, x_n]$, and  $H = (h_1, \dots, h_m), H' = (h_1', \dots, h_m') $ homogeneous sequences of degree 2 in $S$.   Since $\Aut^{\ZZ}S=\GL_n(k)$, the above definition of $st$-equivalence is essentially the same as the one given in \cite[Definition 2.17]{HM}.  In this case, $H\sim _{st}H'$ if and only if  $S/I_H\cong S/I_{H'}$ as graded algebras. 
\end{remark} 

\begin{lemma}
For every fixed $m\in\mathbb{N}$, $\sim_s$, $\sim_t$ and $\sim_{st}$ are equivalence relations on the set of sequences 
$F=(f_1,\cdots, f_m)$ in $S$.
\end{lemma}

\begin{proof}
The relations $\sim_s$ and $\sim_t$ are obviously equivalence relations, and clearly $F\sim_{st}F$. 

If $F = (f_1,\dots, f_m), F' = (f'_1,\dots, f'_m)$ are sequences in $S$ such that $F\sim_{st}F'$, then there exist $(\a_{ij})\in \GL_m(k)$ and $\phi\in \Aut^{\ZZ}S$ such that 
$$
F'= \left(\sum_{i=1}^m\alpha_{i1}\phi(f_i),\cdots, \sum_{i=1}^m\alpha_{im}\phi(f_i)\right)=\left (\phi\left(\sum_{i=1}^m\alpha_{i1}f_i\right),\cdots,\phi\left (\sum_{i=1}^m\alpha_{im}f_i\right )\right) , 
$$
so $F \sim_{s} F'' \sim_t F'$ where $F''=\left(\sum_{i=1}^m\alpha_{i1}f_i,\cdots, \sum_{i=1}^m\alpha_{im}f_i\right)$, hence $F'\sim_{st}F$.  

If $F_1\sim_{st}F_2\sim_{st}F_3$, then there exist $F_4, F_5$ such that $F_1\sim_{t}F_4\sim _sF_2\sim_tF_5\sim_sF_3$.  Since $F_5\sim _{st}F_4$, then $F_4\sim _{st}F_5$ by the above argument,  so there exists $F_6$ such that  $F_1\sim_{t}F_4\sim _tF_6\sim_sF_5\sim_sF_3$, hence $F_1\sim_{st}F_3$.  
\end{proof}

\begin{definition} Let $S$ be a graded algebra, and $H=(h_1, \dots, h_l)$ a homogeneous sequence of degree $d$ in $S$.   For a sequence $G=(g_1, \dots, g_m)$ in $S$, we write a sequence $\bar G:=(\bar g_1, \dots, \bar g_m)$  in $S/I_H$.  
For sequences $F=(f_1, \dots, f_m), F'=(f_1', \dots, f_m')$ of degree $d$ in $S/I_H$,  we define $(F, H)\sim_{st}(F', H)$ in $S$ if there exist sequences $G, G'$ of degree $d$ in $S$ such that $\bar G=F, \bar {G'}=F'$ and that $(G, H)\sim_{st}(G', H)$ in $S$.  
\end{definition} 

\begin{lemma} \label{lem.stSA}  Let $H=(h_1, \dots, h_l)$ be a homogeneous sequence of degree $d$ in $k\<x_1, \dots, x_n\>$, and  $F=(f_1, \dots, f_m), F'=(f_1', \dots, f_m')$ sequences of degree $d$ in $S:=k\<x_1, \dots, x_n\>/I_H$.  If $F\sim _{st}F'$ in $S$, then $(F, H)\sim _{st}(F', H)$ in $k\<x_1, \dots, x_n\>$.  
\end{lemma}

\begin{proof}  Let $G, G'$ be sequences of degree $d$ in $k\<x_1, \dots, x_n\>$ such that $\bar G=F, \bar {G'}=F'$.  

If $F\sim _{s}F'$ in $S$, then there exists $(\a_{ij})\in \GL_m(k)$ such that $\overline {g'_j}=f'_j=\sum_i\a_{ij}f_i=\sum_i\a_{ij}\overline{g_i}$ for $j=1, \dots, m$. Hence $g'_j=\sum_i\a_{ij}g_i+\sum_{i}\c_{ij}h_i$ for some $\c_{ij}\in k$.  Since $\begin{pmatrix} (\a_{ij}) & 0 \\ (\c_{ij}) & E_l \end{pmatrix}\in \GL_{m+l}(k)$, we have $(G, H)\sim _{s}(G', H)$ in $k\<x_1, \dots, x_n\>$. In the matrix notation, 
$$(g_1', \dots, g_m', h_1, \dots, h_l)=(g_1, \dots, g_m, h_1, \dots, h_l)\begin{pmatrix} (\a_{ij}) & 0 \\ (\c_{ij}) & E_l \end{pmatrix}.$$

If $F\sim _{t}F'$ in $S$, then there exists $\bar \phi\in \Aut^{\ZZ} S$ where $\phi \in \Aut^{\ZZ}k\<x_1, \dots, x_n\>=\GL_n(k)$ such that $f'_j=\bar \phi(f_j)$ for $j=1, \dots, m$. Hence, $\overline {g'_j}=f'_j=\bar\phi(f_j)=\bar \phi(\bar {g_j})=\overline {\phi(g_j)}$, and $g_j'=\phi(g_j)+\sum_{r}\c_{rj}h_r$ for some $\c_{rj}\in k$.  Since  $h_r=\sum_{i}\b_{ir}\phi(h_i)$ for some $(\b_{ij})\in \GL_l(k)$, we have
$$g_j'=\phi(g_j)+\sum_{r}\c_{rj}\left(\sum_{i}\b_{ir}\phi(h_i)\right)=\phi(g_j)+\sum_{i}\sum_{r}\b_{ir}\c_{rj}\phi(h_i).$$ 
Since $\begin{pmatrix} E_m & 0 \\  (\b_{ij})(\c_{ij}) & (\b_{ij}) \end{pmatrix}\in \GL_{m+l}(k)$, we have $(G, H)\sim_{st}(G', H)$ in $k\<x_1, \dots, x_n\>$.  In the matrix notation, 
$$(g_1', \dots, g_m', h_1, \dots, h_l)=(\phi(g_1), \dots, \phi(g_m), \phi(h_1), \dots, \phi(h_l))\begin{pmatrix} E_m & 0 \\ (\b_{ij})(\c_{ij}) & (\b_{ij})  \end{pmatrix}.$$
\end{proof}

\begin{definition} Let $S$ be a graded algebra.  For $f\in S$ such that $\deg f=d$, we write $f^\vee:=f_d$.  
We say that a sequence $F=(f_1, \dots, f_m)$ in $S$ is a {\it strongly regular normal sequence} if $F$ is a regular normal sequence in $S$ and $F^\vee:= (f_1^\vee, \dots, f_m^\vee)$ is also a (homogeneous) regular normal sequence in $S$.  
\end{definition} 

\begin{remark} \label{rem.3} 
Let $f, g\in k[x_1, \dots, x_n]\setminus k$. 
\begin{enumerate}
\item{}  It is known that $\gcd(f, g)=1$ if and only if  $(f, g)$ is a regular sequence, so $(f, g)$ is a regular sequence if and only if $(g, f)$ is a regular sequence.  
\item{} If $(f^\vee, g^\vee)$ is a regular (normal) sequence, then $\gcd (f^\vee, g^\vee)=1$, so $\gcd(f, g)=1$, hence $(f, g)$ is automatically a strongly regular (normal)  sequence. This is not the case in the noncommutative setting (see Example \ref{ex.ex} (1)). 
\end{enumerate}
\end{remark} 

\begin{definition} 
We define
\begin{align*}
\cA_{n, m}^\vee & : = \{S/(f_1, \dots, f_m) \mid \text{ $S$ is an $n$-dimensional quantum polynomial algebra and } \\
& \qquad \text  {$F=(f_1, \dots, f_m)$ is a strongly regular normal sequence of degree 2 in $S$}\}/\cong, \\
\cB_{n, m}^\vee & :=\{S/(f_1, \dots, f_m)\in \cA_{n, m}^\vee\mid S=k[x_1, \dots, x_n]\}/\cong, \\ 
\cA_{n, m}^{\vee, st} & : = \{(F, H) \text{ is a linearly independent sequence of degree 2 in $k\<x_1, \dots, x_n\>$} \mid \\
& \qquad \qquad \qquad S=k\<x_1, \dots, x_n\>/I_H\in \cA_{n, 0} \text{ and } S/I_{\bar F}\in \cA_{n,m}^\vee\}/\sim_{st} \\
\cB_{n, m}^{\vee, st} & : = \{(F, H)\in \cA_{n, m}^{\vee, st} \mid  S=k\<x_1, \dots, x_n\>/I_H\cong k[x_1, \dots, x_n]\}/\sim_{st}. 
\end{align*}
An algebra $E$ is called a {\it noncommutative affine pencil of conics} if $E\in \cA_{2,2}^\vee$, and  is called an affine pencil of conics if $E\in \cB_{2,2}^\vee$. 
\end{definition} 

Clearly, we have natural surjections $\cA_{n, m}^{\vee, st}\to \cA_{n, m}^\vee, \cB_{n, m}^{\vee, st}\to \cB_{n, m}^\vee, (\cA_{2, 2}^{\vee, st}\setminus \cB_{2, 2}^{\vee, st})\to  (\cA_{2, 2}^{\vee}\setminus \cB_{2, 2}^{\vee})$.

\subsection{Normalizing automorphisms} 

In order to classify all noncommutative affine pencils of conics, we need to find all strongly regular normal sequences $(f,g)$ of degree 2 for each 2-dimensional quantum polynomial algebra $S \in \cA_{2,0}$ as classified in Example \ref{exm-2dimqpa2}.  
It is not easy to determine if a given element is regular and/or normal in general, however, it is relatively easier to determine if a given ``homogeneous'' element is regular and/or normal.  For example, the following lemma applies to only homogeneous sequences.

\begin{lemma} \label{lem.Tak}  
Let $A$ be a locally finite $\NN$-graded algebra, and $F = (f_1, \dots, f_m)$ a homogeneous normal sequence.
Then $F$ is regular if and only if 
$$
H_{A/I_F}(t)=(1-t^{d_1})\cdots (1-t^{d_m})H_A(t)
$$
where $d_i = \deg f_i$ for $i = 1, \dots, m$. 
\end{lemma} 

\begin{proof} 
This was proved under the assumption that $f_1, \dots, f_m\in A$ are normal in \cite[Lemma 2.28]{HM}, but the same proof works without the assumption. 
\end{proof} 

\begin{remark} \label{rem.2} 
Unlike \cite[Lemma 2.28]{HM}, a homogeneous regular normal sequence is not preserved by permutations (see Example \ref{ex.ex} (2)).   
\end{remark} 

Let $R$ be an algebra. Then $f\in R$ is a regular normal element if and only if there exists a unique algebra automorphism $\nu\in \Aut R$ such that $gf=f\nu(g)$ for $g\in R$.  We call $\nu$ the {\it normalizing automorphism} of $f$.  A regular normal element $f\in R$ is central if and only if $\nu=\id$.  If $A$ is a  graded algebra and $f\in A$ is a homogeneous regular normal element, then $\nu\in \Aut^{\ZZ}A$ is a graded algebra automorphism.   The normalizing automorphism is very useful to find regular normal elements.  We prepare some lemmas below. 

\begin{lemma} \label{lem.reno} 
Let $R$ be an algebra. 
\begin{enumerate}
\item{} 
If there exists a surjective map $\nu: R \to R$ such that $gf=f\nu (g)$ for $g\in R$, then $f\in R$ is a normal element.  In particular, if $R = k\<x_1, \dots, x_n\>/I$ and 
if there exists $(a_{ij})\in \GL_n(k)$ such that $x_if=f\left(\sum_{j=1}^na_{ij}x_j\right)$, then $f$ is normal.
\item{} If $f, g\in R$ are regular normal elements with the same normalizing automorphism, 
then $\a f+\b g\in R$ is a normal element for every $\a, \b\in k$.
\end{enumerate}
\end{lemma}

\begin{remark} \label{rem.1} The converse of Lemma \ref{lem.reno} (2) does not hold in general.  In fact, it is possible that $f, g, f+g\in R$ are regular normal elements such that their normalizing automorphisms are all distinct  (see Example \ref{ex.ex} (3)).
\end{remark} 

\begin{lemma} \label{lem.nag1} 
Let $S$ be a graded algebra and $f=\sum f_i\in S$ a normal element where $f_i\in S_i$.  If there exists $\nu\in \Aut^{\ZZ}S$ such that $gf=f\nu(g)$ for every $g\in S$, then $gf_i=f_i\nu(g)$ for every $g\in S$. Therefore, $f_i$ is a normal element for every $i\in \ZZ$.  

\end{lemma} 

\begin{proof} 
Since $\nu\in \Aut^{\ZZ}S$, we have $\nu(g_j)\in S_j$ for every $g_j\in S_j$. Then,  
$g_jf_{i}=(g_jf)_{i+j}=(f\nu(g_j))_{i+j}=f_i\nu(g_j)$.  
It follows that $gf_i=f_i\nu(g)$ for every $g\in S$, so $f_i$ is a normal element by Lemma \ref{lem.reno} (1). 
\end{proof} 

The following lemma is useful to find a (regular) normal element $f$ of degree 2 in $S\in \cA_{2, 0}$.
\begin{lemma} \label{lem.noce1} 
Let $S=k\<x_1, \dots, x_n\>/I$ be a graded algebra, and $f=\sum f_j\in S$ a regular normal element with the normalizing automorphism $\nu\in \Aut S$ where $f_j\in S_j$.  If  $f^\vee $ is regular, then the following results hold. 
\begin{enumerate}
\item{} $\nu\in \Aut^{\ZZ}S$. 
\item{} $f_j$ is a normal element for every $j$.   
\item{} If $f_j$ is regular, then the normalizing automorphism of $f_j$ is also $\nu$.
\item{} $f\in Z(S)$ if and only if $f_j\in Z(S)$ for every $j$ such that $f_j$ is regular if and only if $f_j\in Z(S)$ for some $j$ such that $f_j$ is regular.
\end{enumerate}
\end{lemma} 

\begin{proof} (1) 
Since $x_if=f\nu(x_i)$ and $f^\vee$ is regular, $\deg \nu(x_i)=1$.  Let $e\in \NN$ be the smallest integer such that $f_e\neq 0$.  Then $\nu(x_i)_0=0$ since $0=f_e\nu(x_i)_0$ and $\nu(x_i)_0 \in k$. This implies that $\nu\in \Aut^{\ZZ}S$. 

(2), (3) Since $gf_j=f_j\nu(g)$ for every $g\in S$ by Lemma \ref{lem.nag1}, $f_j \in S_j$ is a normal element for every $j$,  and the normalizing automorphism of $f_j$ is $\nu$ whenever $f_j$ is regular.

(4) Clear from (3).  
\end{proof} 

By the above lemma, if $(f_1, \dots, f_m)$ is a strongly regular normal sequence, then the normalizing automorphism of $f_i$ is a graded algebra automorphism as long as $f_1, \dots, f_{i-1}$ are homogeneous.  The following lemma is useful to find a regular normal element $g$ of degree 2 in $S/(f)$.

\begin{lemma} \label{lem.noce2} 
Let $S$ be a graded algebra and $f, g\in S$.  If $f$ is a homogeneous normal element, and $\overline {g^\vee}$ is regular in $S/(f)$, then $\bar g$ is regular in $S/(f)$.  
\end{lemma} 

\begin{proof} 
Suppose that there exists $a\in S$ such that $\bar a\neq 0$ but $\bar a\bar g=\overline {ag}=0$ in $S/(f)$.  We may assume that $\overline {a^\vee}\neq 0$ in $S/(f)$.   
Since $f$ is normal in $S$, there exists $b=\sum b_i\in S$ where $b_i\in S_i$ such that $ag=bf$ in $S$.  Since $f$ is homogeneous,  
$(ag)^\vee=(bf)^\vee=(\sum b_if)^\vee= b_jf$ in $S$ for some $j\in \ZZ$, we have $\overline {(ag)^\vee}=0$ in $S/(f)$. Hence, $\overline {a^\vee}\cdot\overline { g^\vee} = \overline{a^\vee g^\vee}=\overline{(ag)^\vee} = 0$ in $S/(f)$. We have $\overline {a^\vee}=0$ in $S/(f)$ since $\overline {g^\vee}$ is regular in $S/(f)$, which is a contradiction. Hence, $\bar  g$ is right regular in $S/(f)$.  We can similarly show that $\bar g$ is left regular in $S/(f)$.  
\end{proof}

\subsection{Classification} 

We now classify noncommutative affine pencils of conics.  Note that since every $S\in \cA_{2, 0}$ is a domain, every non-zero element in $S$ is regular. 

\begin{proposition} List of all (regular) normal elements $f\in S$ of degree 2 up to $t$-equivalence for each $S\in \cA_{2, 0}\setminus \cB_{2, 0}$ is given in Table \ref{tab-f}.
\begin{table}[h]
\caption{List of (regular) normal elements in $S$ up to $t$-equivalence.} \label{tab-f}
\begin{tabular}{|c|c|}
\hline
$S$                & $f$                   \\
\hline
 $k_J[x, y]$                & $y^2$                   \\ 
\hline 
$k_{\l}[x, y] \; (\l\neq 0, \pm 1)$                & $x^2, y^2, xy$                   \\ 
\hline 
$k_{-1}[x, y]$                & $x^2, x^2+1, x^2+y^2, x^2+y^2+1, xy$                   \\ 
\hline
\end{tabular}
\end{table}
\end{proposition} 

\begin{proof} (Sketch) Let $S\in \cA_{2, 0}\setminus \cB_{2, 0}$ and $f\in S$ a (regular) normal element of degree 2 in $S$.   Since 
$f_0\in Z(S)$, if $f_0\neq 0$, then $f\in Z(S)$ by Lemma \ref{lem.noce1} (4).  Moreover, if $f_1\neq 0$, then $f_1$ is a (regular) normal element having the same normalizing automorphism as that of  $f_2$ by Lemma \ref{lem.noce1} (3). If $S$ is not commutative, then it is not difficult to find all homogeneous (regular) normal elements of degree 1 and 2 in $S$, and to compute their normalizing automorphisms, to conclude $f_1=0$.  
\end{proof} 

\begin{example} Let $S=k_{\l}[x, y]\in \cA_{2, 0}\setminus \cB_{2, 0}$ where $\l \neq 0, 1$.  
We can check that the only homogeneous (regular) normal elements of degree 1 are $x, y$ up to scalar. 
If $\l \neq -1$, then we can check that the only homogeneous (regular) normal elements of degree 2 are $x^2, y^2, xy$ up to scalar.  We can also check that the normalizing automorphisms of $x, y, x^2, y^2, xy$ are all distinct and not the identity, so we can conclude that the only  (regular) normal elements of degree 2 are $x^2, y^2, xy$ up to scalar by Lemma \ref{lem.noce1} (3).   

If $\l=-1$, then we can check that the only homogeneous (regular) normal elements of degree 2 are $\a x^2+\b y^2, xy$ up to scalar where $\a, \b\in k$.  We can also check that the normalizing automorphisms of $x, y, \a x^2+\b y^2, xy$ are all distinct and $\a x^2+\b y^2$ is the only central elements among them, so we can conclude that the only  (regular) normal elements of degree 2 are $\a x^2+\b y^2+\c, xy$ up to scalar where $\a, \b, \c\in k$ by Lemma \ref{lem.noce1} (3), (4).
Since 
$$\Aut^{\ZZ}(S)=\left \{\begin{pmatrix} a & 0 \\ 0 & d \end{pmatrix}, \begin{pmatrix} 0 & b \\ c &  0 \end{pmatrix}\mid a, b, c, d\in k, ad\neq 0, bc\neq 0\right\},$$ 
we can show that $f$ is $x^2, x^2+1, x^2+y^2, x^2+y^2+1, xy$ up to $t$-equivalence.    
\end{example}

\begin{theorem} \label{thm.Tb1} 
Table \ref{tab-fgh} consists of the following information: 
\begin{enumerate}
\item{} List of 2-dimensional quantum polynomial algebras $S=k\<x, y\>/(h)
\in \cA_{2, 0}\setminus \cB_{2, 0}$. 
\item{} List of (regular) normal elements $f\in S$ of degree 2 up to $t$-equivalence for each $S$. 
\item{} List of elements $g\in S/(f)$ such that $(f, g)$ is a strongly regular normal sequence of degree 2 in $S$ up to $st$-equivalence of  $(f, g)$ in $S$ for each pair $(S, f)$ such that  $S/(f)$ is not commutative.  
\item{} Name of the sequence $(f, g, h)$ in $k\<x, y\>$.  
\end{enumerate}    
(Since $(y^2, x^2, xy-\l yx)\sim _{s}(x^2, y^2, xy-\l yx)$ in $k\<x, y\>$, we delete the pair $(k_{\l}[x, y], y^2)$ from Table \ref{tab-fgh}.) 
\begin{table}[h]
\caption{List of regular normal elements in $S/(f)$.} 
\label{tab-fgh}
\begin{tabular}{|l|l|l|l|}
\hline
$S=k\<x, y\>/(h)$                & $f$         & $g$                            &  $(f, g, h)$ \\ \hline
$k_J[x,y]$         & $y^2$       & $S/(f)$: comm.                          & $F_1$          \\ \hline
$k_{\lambda}[x,y]$ & $x^2$       & $y^2$                          & $F_2(\l)$          \\ \cline{2-4} 
$(\l\neq 0, \pm 1)$          & $yx$        & $S/(f)$: comm.                          & $F_3$          \\ \hline
\multirow{11}{*}{$k_{-1}[x,y]$}     & \multirow{3}{*}{$x^2$}      & $y^2$                          & $F_4$          \\ \cline{3-4} 
                   &             &  $y^2+yx$                    & $F_5$          \\ \cline{3-4} 
                   &             & $y^2+1$                        & $F_6$          \\ \cline{2-4} 
                   & \multirow{2}{*}{$x^2+1$}     & $y^2$                          & $F_7$          \\ \cline{3-4} 
                   &             & $y^2+1$                        & $F_8$          \\ \cline{2-4} 
                   & \multirow{3}{*}{$x^2+y^2$}   & $x^2$                          & $F_9$          \\ \cline{3-4} 
                   &             & $x^2+1$                        & $F_{10}$         \\ \cline{3-4} 
                   &             &  $x^2+\a yx \;\; (\a  \neq 0,\pm 1)$ & $F_{11}(\a)$         \\ \cline{3-4} 
                   &             &  $yx$ & $F_{12}$        \\ \cline{2-4}
                   & \multirow{2}{*}{$x^2+y^2+1$} & $x^2 + \a \;\; (\a=0,1) $                          & $F_{13}(\a)$         \\ \cline{3-4} 
                   &  & $x^2 + \a \;\; (\a \neq 0,1)$                          & $F_{14}(\a)$         \\ \cline{3-4}
                   &  & $yx$                          & $F_{15}$         \\ \cline{2-4}
                   & $yx$        & $S/(f)$: comm.                          & $F_{16}$        \\ \hline
\end{tabular}
\end{table}
\end{theorem}

\begin{proof} 
(Sketch) We will prove by the following steps: 

{\bf Step 1:}  Show that every $\{f, h\}$ in Table \ref{tab-fgh} is a noncommutative Gr\"obner basis for the two-sided ideal $(f, h)$ of $k\<x, y\>$. 

{\bf Step 2:} Fix a $k$-basis for each $S/(f)=k\<x, y\>/(f, h)$ using  the noncommutative Gr\"obner basis obtained in Step 1.  

{\bf Step 3:}  For an element $g\in S/(f)$ of degree 2, write $xg, yg, gu, gv$ where $u, v\in S/(f)$ are elements of degree 1 as linear combinations of the $k$-basis for $S/(f)$ obtained in Step 2 and compare the coefficients of the $k$-basis for the pairs $(xg, gu)$ and $(yg, gv)$ to find all normal elements of degree 2 in $S/(f)$ by Lemma \ref{lem.reno} (1).

{\bf Step 4:} For $g$ obtained in Step 3, check that $(f, g)$ is a strongly regular normal sequence. 
\end{proof}

\begin{example} \label{ex.st4} 
Let $S/(f)=k_{-1}[x, y]/(x^2)$.   Using the order $y<x$, we can check that $\{x^2, xy+yx\}$ is a  noncommutative Gr\"obner basis for the two-sided ideal $(x^2, xy+yx)$ of $k\<x, y\>$, so we can compute that $k\<x, y\>/(x^2, xy+yx)=S/(f)$ has a $k$-vector space basis $\{1, x, y, yx, y^2, y^2x, y^3, \dots \}$.  Suppose that $(f, g)$ is a strongly regular normal sequence of degree 2 in $S=k_{-1}[x, y]$.  Since $(f^\vee, g^\vee)=(x^2, g_2)$ is a homogeneous regular normal sequence of degree 2 in $S$, $g_2\neq yx$, so we may write $g_2=\c yx+y^2$ for some $\c\in k$ up to scalar.  Since
\begin{align*} 
xg_2 & =x(\c yx+y^2)=y^2x=(\c yx+y^2)x=g_2x, \\
yg_2 & =y(\c yx+y^2)=\c y^2x+y^3=(\c yx+y^2)(\c x+y)=g_2(2\c x+y), 
\end{align*}
$g_2$ is a normal element with the normalizing automorphism $\nu=\begin{pmatrix} 1 & 0 \\ 2\c & 1 \end{pmatrix}$, and $g_2$ is central if and only if $\c=0$. By Lemma \ref{lem.noce1} (3), the normalizing automorphism of $g$ is also $\nu\in \Aut^{\ZZ}S/(f)$, so  $ug_1 = g_1 \nu(u)$ for every $u\in S$ by Lemma \ref{lem.nag1}.
If $g_1=\a x+\b y$ where $\a, \b\in k$, then
\begin{align*} 
&-\b yx = x(\a x+\b y) = xg_1  = g_1\nu(x) = (\a x+\b y)x=\b yx, \\
&\a yx+\b y^2 = y(\a x+\b y)= yg_1  = g_1\nu(y)  =(\a x+\b y)(2\c x+y)=(2\b\c-\a)yx+\b y^2,
\end{align*}
so $\a=\b=0$. It follows that $g_1=0$, so $g=y^2+\c yx$ or $g=y^2+\d$ for some $\c, \d\in k$. It is now easy to see that $g$ is $y^2, y^2+1, y^2+yx$ up to $t$-equivalence in $S$.  Since 
$$(k_{-1}[x, y]/(x^2, y^2))^!\cong k[x, y], \qquad (k_{-1}[x, y]/(x^2, y^2+yx))^!\cong k_J[x, y]\in \cA_{2, 0},$$ 
and
$$H_{k_{-1}[x, y]/(x^2, y^2)}(t)=H_{k_{-1}[x, y]/(x^2, y^2+yx))}(t)=(1+t)^2=(1-t^2)^2/(1-t)^2$$ 
by Theorem \ref{thm.qpak}, $g_2\in S/(f)$ is regular by Lemma \ref{lem.Tak}.  By Lemma \ref{lem.noce2}, $g\in S/(f)$ is regular,  so $(f, g)$ is a strongly regular normal sequence.
\end{example}

\begin{example} \label{ex.ex} Example \ref{ex.st4} provides some counter-examples. 
\begin{enumerate}
\item{} For a sequence $(f, g)=(x^2, y^2+y)$ in $ k_{-1}[x, y]$, $(f^\vee, g^\vee)=(x^2, y^2)$ is a regular normal sequence, but $(f, g)$ is not a (strongly) regular normal sequence (see Remark \ref{rem.3} (2)).
\item{} $(x^2, y^2+yx)$ is a regular normal sequence in $ k_{-1}[x, y]$, but $(y^2+yx, x^2)$ is not a regular normal sequence in $ k_{-1}[x, y]$ (see Remark \ref{rem.2}).
\item{}  $y^2+yx, y^2-yx\in k_{-1}[x, y]/(x^2)$ are regular normal (non-central) elements with the distinct normalizing automorphisms, but $(y^2+yx)+(y^2-yx)=2y^2\in k_{-1}[x, y]/(x^2)$ is a regular central element (see Remark \ref{rem.1}).  
\end{enumerate}  
\end{example}

\begin{theorem} \label{thm.ANC}\label{lem.scn}  
Let $(f, g, h)\in \cA_{2, 2}^{\vee, st}$ so that $S=k\<x, y\>/(h)\in \cA_{2, 0}$ and $E=S/(f, g)\in \cA_{2, 2}^\vee$. 
\begin{enumerate}
\item{} $E$ is not commutative if and only if $(f, g, h)$ is st-equivalent in $k\<x, y\>$ to one of the following: 
$$(x^2, y^2, xy-\l yx), (x^2, y^2+yx, xy+yx), (x^2,y^2+1, xy+yx), (x^2+1,y^2+1, xy+yx)$$
where $\l\neq 0, 1$. 
\item{}  $E$ is commutative if and only if there exists $(f', g', xy-yx)\in \cB_{2, 2}^{\vee, st}$ such that $$(f, g, h)\sim_{st}(f', g', xy-yx)$$ in $k\<x, y\>$. 
\end{enumerate}
\end{theorem}  

\begin{proof} 
(1) If $E$ is not commutative, then $S/(f)$ is not commutative, so $F_1, F_3, F_{12}, F_{15}, F_{16}$ in Table \ref{tab-fgh} are excluded.  It is easy to see that $F_4\sim_{st}F_9, F_6\sim _{st}F_7\sim_{st}F_{13}(\a), F_8\sim _{st}F_{10}$  in $k\<x, y\>$  in Table \ref{tab-fgh}. For $F_{14}(\a)$, we have
$F_{14}(\a) \sim_{st} (x^2+\a , y^2+ (1-\a), xy+ yx) \sim_{st} F_8$. 
If $\phi=\begin{pmatrix} 1 & 1 \\ 1 & -1 \end{pmatrix}\in \GL_2(k)=\Aut^{\ZZ}(k\<x, y\>)$, then 
\begin{align*}
& \phi(x^2+y^2)=(x+y)^2+(x-y)^2=2(x^2+y^2), \\
& \phi(x^2+\a xy)=(x+y)^2+\a (x+y)(x-y)=(1+\a )x^2+(1-\a )xy+(1+\a )yx+(1-\a )y^2, \\
& \phi(xy+yx)=(x+y)(x-y)+(x-y)(x+y)=2(x^2-y^2), 
\end{align*}
so $(x^2+y^2, x^2+\a xy, xy+yx)\sim_{st} (x^2, y^2, xy-\l yx)$  in $k\<x, y\>$  where $\l=(\a +1)/(\a -1)$, that is, $F_{11}(\a)\sim_{st}F_2( (\a+1)/(\a -1))$  in $k\<x, y\>$.  
It is easy to see that none of these algebras $E$ is commutative.

 (2) Suppose that $E$ is commutative.  If $S=k[x, y]$, then the result is trivial. If $S\neq k[x, y]$, then it is enough to check it for $F_1, F_3, F_{12}, F_{15}, F_{16}$ in Table \ref{tab-fgh}.  
Since 
\begin{align*}
&F_1\sim_s(y^2,g, xy-yx),  \, F_3\sim_sF_{16}\sim_s(xy, g, xy-yx), \\
&F_{12} \sim_s (x^2 + y^2 ,xy, xy-yx), \, F_{15} \sim_s (x^2 + y^2 + 1, xy, xy-yx)
\end{align*}
in $k\<x, y\>$, the result follows.
\end{proof} 
 
We restate Theorem \ref{thm.ANC} in the following form.  

\begin{corollary} \label{cor.ANC}
Every $(f, g, h)\in \cA_{2, 2}^{\vee, st}\setminus \cB_{2, 2}^{\vee, st}$ is st-equivalent  in $k\<x, y\>$ to one of the following:
$$(x^2, y^2, xy-\l yx), (x^2,xy+y^2, xy+yx), (x^2,y^2+1, xy+yx), (x^2+1,y^2+1, xy+yx)$$
where $\l\neq 0, 1$. 
\end{corollary} 

\begin{proof}  Let $E=k\<x, y\>/(f, g, h)\in \cA_{2, 2}^\vee$.  
If $E$ is not commutative, then this follows from Theorem \ref{thm.ANC} (1).  If $E$ is commutative, then 
$(f, g, h)\in \cB_{2, 2}^{\vee, st}$, that is, $(f, g, h)\not \in \cA_{2, 2}^{\vee, st}\setminus \cB_{2, 2}^{\vee, st}$ by Theorem \ref{thm.ANC} (2).  
\end{proof} 

The above theorem (corollary) shows that $E\in \cA_{2, 2}^{\vee}$ is commutative if and only if $E\in \cB_{2, 2}^{\vee}$, which is not obvious from the definition. 

\section{$4$-dimensional Frobenius algebras}

In this section, we give complete classifications of 4-dimensional Frobenius algebras up to isomorphism and up to derived equivalence. 

\begin{definition} \label{def-fro}
A finite dimensional algebra $R$ is called a {\it Frobenius} algebra if there exists a bilinear form 
$
(-,-): R \times R \to k
$ 
satisfying the following conditions:
\begin{itemize}
\item[(1)] Associative: $(ab,c) = (a,bc)$ for all $a,b,c \in R$. 
\item[(2)] Nondegenerate: for any $a \in R$, there are $b, c\in R$ such that $(a,b) \neq 0$ and $(c,a) \neq 0$. 
\end{itemize}
\end{definition} 
 
The following lemmas are useful. 
 
\begin{lemma}  [{\cite[Proposition IV 2.4]{SY}}] \label{lem-profro}
If $R = R_1\times \cdots \times R_n$ is a direct product of finite dimensional algebras. Then $R$ is Frobenius if and only if $R_i$ is Frobenius for every $i=1, \dots , n$.
\end{lemma}

We say a quiver $Q$ is {\it path connected} if for any two vertices $i$ and $j$, there are a path from $i$ to $j$ and a path from $j$ to $i$. 
The lemma below follows from {\cite[Corollary 3.4]{Gr}} since every Frobenius algebra is (left and right) self-injective by \cite[Proposition IV 3.8]{SY}.  

\begin{lemma} 
\label{lem.Gr} 
If $R = kQ/I$ is a Frobenius algebra where $Q$ is a connected quiver and $I$ is an admissible ideal, then $Q$ is path connected. 
\end{lemma} 
 
\begin{lemma} \label{lem-3dimfro}
Every Frobenius algebra up to dimension $3$ is commutative.
\end{lemma}

\begin{proof}  
	Let $E$ be an algebra up to dimension 3. If $E$ is not commutative, then it is known that $E\cong  k(\xymatrix{ 1 \ar[r] & 2}) $, which is not Frobenius  by Lemma \ref{lem.Gr}.  
\end{proof} 

There is a close relationship between quantum polynomial algebras and Frobenius algebras. 
\begin{theorem} [{\cite[Proposition 5.10]{Sm96}}]  \label{thm.FAS} 
Let $S$ be a Koszul algebra of finite global dimension.  Then  $S$ satisfies Gorenstein condition if and only if $S^!$ is a Frobenius algebra.   In particular, if $S$ is a quantum polynomial algebra, then $S^!$ is a Frobenius algebra.  
\end{theorem} 

\begin{theorem}\label{thm-4dimfro}
Every $4$-dimensional Frobenius algebra is isomorphic to exactly one of the algebras listed in Table \ref{tab-4dimfro} where $k_\l[x,y]/(x^2,y^2)\cong k_{\l'}[x, y]/(x^2, y^2)$ if and only if $\l'=\l^{\pm 1}$.
\begin{center}
\begin{table}[h]
\begin{threeparttable}
\centering
\caption{$4$-dimensional Frobenius algebras.} \label{tab-4dimfro}
\begin{tabular}{|c|}
\hline
{\rm $\cG_2$}  \\ \hline 
$k^4$, $k^2\times k[x]/(x^2)$, $k\times k[x]/(x^3)$, $(k[x]/(x^2))^2$, $k[x]/(x^4)$, $k[x,y]/(x^2,y^2) $ \\ 
\hline \hline
{\rm $\cF_2\setminus \cG_2$}  \\ \hline  
$M_2(k)$, $k (\xymatrix{
1 \ar@<0.8ex>[r]^-x & 2 \ar@<0.7ex>[l]^-y
}) \; /(xy, yx)$, 
$k_{-1}[x,y] /(x^2, y^2+yx)$, $k_\l[x,y]/(x^2,y^2)$ where $\l\neq 0, 1$ \\\hline
\end{tabular}
\end{threeparttable}
\end{table}
\end{center}
\end{theorem}

\begin{proof} The classification of $\cG_2$ is given in  \cite[Corollary 4.10]{H}, so we will classify $\cF_2\setminus \cG_2$.  If $E\in \cF_2$ is a direct product of two algebras, then $E$ is commutative by Lemma  \ref{lem-profro} and Lemma \ref{lem-3dimfro}, so we exclude this case.  In particular, we exclude the case that $E$ is a basic algebra and the quiver of $E$ is not connected. 

We divide the proof into $4$ cases in terms of the dimension of $E / \rad E$. 

(a) $\dim_k E / \rad E = 4$:   Since $\rad E = 0$, $E$ is semisimple (so that $E$ is a Frobenius algebra by \cite[Proposition IV 6.7]{SY}), so either $E\cong k^4$ (commutative) or $E\cong M_2(k)$. 

(b) $\dim_k E / \rad E = 3$:  Since $E / \rad E \cong k^3$, $E$ is a basic algebra.  The quiver $Q$ of $E$ has $3$ vertices with one arrow, so $Q$ is not connected. 

(c) $\dim_k E / \rad E = 2$:  Since $E / \rad E \cong k^2$, $E$ is a basic algebra.  Since the quiver $Q$ of $E$ is connected, it has two vertices with two arrows, 
so $Q$ is one of the following.
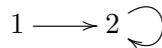
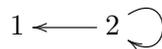
\begin{itemize}
\item[(i)] $$
\xymatrix{
1 \ar[r] & 2 \ar@(ur,dr)
} $$
\item[(ii)] $$
\xymatrix{
1 & 2 \ar@(ur,dr) \ar[l] 
} $$
\item[(iii)] $$
\xymatrix{
1 \ar@<0.8ex>[r] \ar@<-0.8ex>[r]& 2
}
$$
\item[(iv)] $$
\xymatrix{
1 \ar@<0.8ex>[r] & 2 \ar@<0.7ex>[l]
}
$$
\end{itemize}
By Lemma \ref{lem.Gr}, 
(iv) is the only possible quiver for a Frobenius algebra. Since $\dim_k E = 4$, 
$$
E \cong 
k (\xymatrix{
1 \ar@<0.8ex>[r]^-x & 2 \ar@<0.7ex>[l]^-y
}) \; /(xy, yx),
$$
which is known to be a Frobenius algebra (\cite[Example  IV 7.5]{SY}).
		
(d) $\dim_k E / \rad E = 1$:  Since $E / \rad E \cong k$, $E$ is a basic algebra. 
The quiver $Q$ of $E$ has one vertex with $i$ arrows where $i = 1, 2, 3$. 
If $i = 1$, then $E \cong k[x]/I$ is commutative. If $i = 3$, then
$$
E \cong k\langle x,y,z\rangle/(x^2,y^2,z^2, xy, yx, xz, zx, yz, zy) = k [x,y,z]/(x^2,y^2,z^2, xy, xz, yz)
$$ 
is commutative. If $i = 2$, then $E \cong T/I$ for some admissible ideal $I\lhd T=k\<x, y\>$.  Since $I\subset T_{\geq 2}$, 
$$T_0\oplus T_1\oplus (T_2/(T_2\cap I))\subset T/I$$ as vector spaces.  Since $\dim _k(T_0\oplus T_1)=3$, 
$\dim _k(T_2/(T_2\cap I))\leq 1$, so $\dim _k(T_2\cap I)\geq 3$.  It follows that there are a linearly independent elements $h_1, h_2, h_3\in T_2\cap I$ and a surjective algebra homomorphism $\phi:R=T/(h_1, h_2, h_3)\to E$. 
Recall that, for quadratic algebras $A, A'$,  $A\cong A'$ as graded algebras if and only if $A^!\cong {A'}^!$ as graded algebras (cf. \cite [Lemma 4.1]{MU2}). Since $(T/(h_1, h_2, h_3))^!\cong T/(h)$ where $0\neq h\in T_2$ is isomorphic as graded algebras to one of the graded algebras in Example \ref{exm-2dimqpa2}, 
$T/(h_1, h_2, h_3)$ is isomorphic to one of the following graded algebras: 
\begin{itemize}
\item[(i)] $R_1=(k\<x, y\>/(x^2))^!\cong k[x, y]/(xy, y^2)$,  
\item[(ii)] $R_2=(k\<x, y\>/(xy))^!\cong k\<x, y\>/(x^2, yx, y^2)$,
\item[(iii)] $R_3=k_J[x, y]^!\cong k_{-1}[x, y]/(x^2, y^2+yx)$,   
\item[(vi)] $R_4(\l)=k_{-\l^{-1}}[x, y]^!\cong k_{\l}[x, y]/(x^2, y^2), \, \l\neq 0$.  
\end{itemize} 
We exclude $R_1$ since it is commutative.  By Example \ref{exm-2dimqpa2}, $k\<x, y\>/(xy), k_J[x, y], k_{\l}[x, y]\; (\l\neq 0)$ are Koszul algebras of global dimension 2 with Hilbert series $1/(1-t)^2$, so $R_2, R_3, R_4(\l)$ are 4-dimensional algebras by Theorem \ref{thm.KoHi}, hence $\phi$ is an isomorphism.  Moreover, by Example \ref{exm-2dimqpa2}, $k\<x, y\>/(xy)$ does not satisfy Gorenstein condition, while $k_J[x, y], k_{\l}[x, y]\; (\l\neq 0)$ satisfy Gorenstein condition, so we can conclude that $R_2$ is not Frobenius, while $R_3, R_4(\l)$ are Frobenius by Theorem \ref{thm.FAS}.

By  Example \ref{exm-2dimqpa2}, any two algebras in Table \ref{tab-4dimfro} are not isomorphic (even as ungraded algebras by \cite{BZ}) to each other except that $k_\l[x,y]/(x^2,y^2)\cong k_{\l'}[x, y]/(x^2, y^2)$ if and only if $\l'=\l^{\pm 1}$.
\end{proof}

We will now classify 4-dimensional Frobenius algebras up to derived equivalence. 

\begin{lemma}  [{\cite[Proposition 9.2]{R}}] \label{lem.ee'c} 
Let $R$, $R'$ be rings.  If $\cD^b(\mod R)\cong \cD^b(\mod R')$, then $Z(R)\cong Z(R')$.  
\end{lemma}

\begin{lemma}[{\cite[Corollary 5.8]{CJ}}] \label{lem-nptme}
Let $R$, $R'$ be finite dimensional algebras such that $R$ is self-injective and its Nakayama permutation is transitive.  
Then $\cD^b(\mod R) \cong \cD^b(\mod R')$ if and only if $\mod R\cong \mod R'$.
\end{lemma}

\begin{lemma}[{\cite[Proposition  II 6.20]{SY}}] \label{lem-baiffme}
Let $R$, $R'$ be basic finite dimensional algebras. Then $\mod R\cong \mod R'$ if and only if $R\cong R'$. 
\end{lemma}

\begin{remark}\label{rem.tN} 
The Nakayama permutation of $k (\xymatrix{
1 \ar@<0.8ex>[r]^-x & 2 \ar@<0.7ex>[l]^-y }) \; /(xy, yx)$ is transitive by {\cite[Example  IV 7.5]{SY}}.  
\end{remark}

\begin{theorem} \label{thm-de4dimfro}
For $4$-dimensional Frobenius algebras $E$, $E'$, the following are equivalent:
\begin{itemize}
\item [(1)] $E \cong E'$.
\item [(2)] $\mod E\cong \mod E'$.  
\item [(3)] $ \cD^b(\mod E)\cong \cD^b(\mod E')$.
\end{itemize}
\end{theorem}

\begin{proof} It is enough to show (3) $\Rightarrow$ (1).  

Suppose that $E$ is commutative.  If $\cD^b(\mod E)\cong \cD^b(\mod E')$, then $Z(E')\cong Z(E)=E$ by Lemma \ref{lem.ee'c}.  Since $\dim _kZ(E')=\dim _kE=4=\dim_kE'$, $Z(E')=E'$.  

If $E$ is not commutative, then either $E$ is basic, or else $E\cong M_2(k)$ by Theorem \ref{thm-4dimfro}. If $E$ is basic, then it is a basic self-injective algebra with transitive Nakayama permutation (\cite[Proposition IV 3.8]{SY}, Theorem \ref{thm-4dimfro}, Remark \ref{rem.tN}), so $\cD^b(\mod E)\cong \cD^b(\mod E')$ if and only if $E \cong E'$
 by Lemma \ref{lem-nptme} and Lemma \ref{lem-baiffme}.  Suppose $E\cong M_2(k)$.  If $\cD^b(\mod E')\cong \cD^b(\mod E)\cong \cD^b(\mod k)$, then $E'\cong M_2(k)\cong E$, hence the result.
\end{proof}

\begin{remark}  For $A, A'\in \cA_{3, 1}$, Theorem \ref{thm-de4dimfro} shows that $\uCM^{\ZZ}(A)\cong \uCM^{\ZZ}(A')$ if and only if  $C(A)\cong C(A')$ (see \cite[Proposition 5.2]{SV}). 
\end{remark} 

\section{Main results} \label{sec-main}

In this last section, we prove our main result, the classification of noncommutative affine pencils of conics is the same as the classification of $4$-dimensional Frobenius algebras.  We prove it in the commutative case and in the noncommutative case separately.   In the commutative case, we use the techniques of homogenization and dehomogenization.

\begin{definition} Let $S=k[x_1, \dots, x_n]$.  
\begin{enumerate}
\item{} (homogenization) For $f\in S$ such that $\deg f=d$, we define 
$$
f^z:=f(x_1z^{-1}, \dots, x_nz^{-1})z^d\in S[z]_d. 
$$  
For a sequence $F = (f_1, \dots, f_m)$ in  $S$, we define 
$$
F^z:=((f_1)^z, \dots, (f_m)^z),
$$ 

and $\sH^z(F):=S[z]/I_{F^z}$.  
\item{}  (dehomogenization) For $f\in S[z]$, we define 
$$
f_z:=f(x_1, \dots, x_n, 1)\in S.
$$  

For a sequence $F = (f_1, \dots, f_m)$ in $S[z]$, we define
$$
F_z = ((f_1)_z, \dots, (f_m)_z), 
$$

and $\sD_z(F):=S/I_{F_z}$.  

\end{enumerate} 
\end{definition}  

If $f=\sum_{i=0}^df_i\in S$ such that $\deg f=d$  
where $f_i\in S_i$, then $f^z=\sum_{i=0}^df_iz^{d-i}\in S[z]_d$.   
The following lemma is well-known (and easy to check): 

\begin{lemma} \label{lem.ddz}  Let $S=k[x_1, \dots, x_n]$.  
\begin{enumerate}
\item{} For a sequence  $F $ in  $S$, $(F^z)_z=F$.
\item{}
Let $H$ be a homogeneous sequence in $S[z]$, and $B=S[z]/I_H$ a graded algebra.  If $z\in B_1$ is regular, then $\sD_z(H)\cong B[z^{-1}]_0$.  
\end{enumerate}
\end{lemma} 

By the above lemma, we make the following definition by abuse of notations.  

\begin{definition} 
For a commutative graded algebra $B$ finitely generated in degree 1 and a regular element $z\in B_1$, we define $\sD_z(B):=B[z^{-1}]_0$.  
\end{definition} 

\begin{lemma} \label{lem.DHi} 
For a sequence $F$ in $S=k[x_1, \dots, x_n]$,
 if $z\in \sH^z(F)$ is regular, then 
$\sD_z(\sH^z(F))\cong S/I_F$. 
\end{lemma} 

\begin{proof} By Lemma \ref{lem.ddz},
\begin{equation*}
\sD_z(\sH^z(F))\cong \sD_z(F^z)=S/I_{(F^z)_z}=S/I_F.  \qedhere
\end{equation*} 
\end{proof}

We denote by $\Kdim B$ the Krull dimension of  a commutative graded algebra $B$.

\begin{lemma} \label{lem.wbwd}
For $B=k[x_1, \dots, x_n]/(f_1, \dots, f_{n-1})\in \cB_{n, n-1}$, there exists a regular element $z\in B_1$. 
\end{lemma} 

\begin{proof} Since $B$ is a graded Cohen-Macaulay algebra such that $\depth B=\Kdim B=1$, there exists a homogeneous regular element $f_n\in B$ so that $(f_1, \dots, f_{n-1}, f_n)$ is a homogeneous regular sequence for $k[x_1, \dots, x_n]$, so $f_1, \dots, f_{n-1}, f_n$ is a homogeneous system of parameters for $k[x_1, \dots, x_n]$. Since $x_1, \dots, x_n$ is also a homogeneous system of parameters of degree 1 for $k[x_1, \dots, x_n]$, there exists $z=\sum \l_ix_i$ where $\l_i\in k$ such that  $f_1, \dots, f_{n-1}, z$ is a homogeneous system of parameters for $k[x_1, \dots, x_n]$ by \cite[Lemma 2.3.3]{St}.  It follows that $(f_1, \dots, f_{n-1}, z)$ is a homogeneous regular sequence for $k[x_1, \dots, x_n]$ by \cite[Theorem 17.4]{M}, so $z\in B_1$ is a regular element.  
\end{proof}

\begin{lemma} \label{lem.2} Let $B$ be a commutative graded algebra finitely generated in degree 1 such that $\Kdim B=1$.  If $f, f'\in B$ are homogeneous regular elements, then $B[f^{-1}]_0\cong B[{f'}^{-1}]_0$.  In particular, if $z, z'\in B_1$ are regular, then $\sD_z(B)\cong \sD_{z'}(B)$.  
\end{lemma} 

\begin{proof} Since $\Kdim B/(f)=\Kdim B/(f')=0$, 
$\Spec B[f^{-1}]_0\cong \Proj B\cong \Spec B[{f'}^{-1}]_0$, so $B[f^{-1}]_0\cong B[{f'}^{-1}]_0$. 
\end{proof}  

By the above lemma, we may write $\sD(B)$ in place of $\sD_z(B)$ if $\Kdim B=1$.   

\begin{lemma} \label{lem.3} 
If $A\in \cC_{3, 1}$, then $C(A)\cong \sD(A^!)$.  
\end{lemma} 

\begin{proof} If $A\in \cC_{3, 1}$, then $A^!\in \cB_{3, 2}$ by Theorem \ref{thm.HMd}, so there exists a regular element $z\in A^!_1$ by Lemma \ref{lem.wbwd}.  Since $\Kdim A^!=1$ and $f^!\in A^!_2$ is regular by Lemma \ref{lem.asf},  $C(A):=A^![f^{-1}]_0\cong A^![z^{-1}]_0=:\sD(A^!)$ by Lemma \ref{lem.2}.   
\end{proof}

\begin{theorem} \label{thm.2} 
If $E=k[x, y]/(f, g)\in \cB_{2, 2}^\vee$, then $\sH^z(f, g)=k[x, y, z]/(f^z, g^z)\in \cB_{3, 2}$, and $\sD_z(\sH^z(f, g))\cong E$. 
\end{theorem}

\begin{proof} Since 
$k[x, y, z]/(f^z, g^z, z)\cong k[x, y]/(f^\vee, g^\vee)$ and $(f^\vee, g^\vee)$ is a homogeneous regular sequence of degree 2 in $k[x, y]$, 
$$H_{k[x, y, z]/(f^z, g^z, z)}(t)=H_{k[x, y]/(f^\vee, g^\vee)}(t)=(1-t^2)^2/(1-t)^2=(1-t^2)^2(1-t)/(1-t)^3$$ 
by Lemma \ref{lem.Tak}.  
Since $(f^z, g^z, z)$ is a homogeneous sequence of elements of degree 2, 2, 1 in $k[x, y, z]$, $(f^z, g^z, z)$ is a regular sequence in $k[x, y, z]$ by Lemma \ref{lem.Tak} again.  It follows that  $(f^z, g^z)$ is a homogeneous regular sequence of degree 2 in $k[x, y, z]$, so $\sH^z(f, g) =k[x, y, z]/(f^z, g^z)\in \cB_{3, 2}$.  

Since $z\in \sH^z(f, g)_1$ is regular by the above argument, $\sD_z(\sH^z(f, g))\cong k[x, y]/(f, g)=E$ by Lemma \ref{lem.DHi}. 
\end{proof}

\begin{theorem} \label{thm.1} 
If $B=k[x, y, z]/(f, g)\in \cB_{3, 2}$ and $z\in B_1$ is regular, then $\sD_z(B)\cong k[x, y]/(f_z, g_z)\in \cB_{2, 2}^\vee$. 
\end{theorem}

\begin{proof} Since $z\in B_1$ is regular, 
$k[x, y, z]/(f, g, z)\cong k[x, y]/((f_z)^\vee, (g_z)^\vee)$.  Since $(f, g, z)$ is a regular sequence of elements of degree 2, 2, 1 in $k[x, y, z]$, 
$$H_{k[x, y]/((f_z)^\vee, (g_z)^\vee)}(t)=H_{k[x, y, z]/(f, g, z)}(t)=(1-t^2)^2(1-t)/(1-t)^3=(1-t^2)^2/(1-t)^2,$$ 
so $(f_z^\vee, g_z^\vee)$ is a regular sequence in $k[x, y]$ by Lemma \ref{lem.Tak}.  By Remark \ref{rem.3},  $(f_z, g_z)$ is a strongly regular normal sequence in $k[x, y]$, so $\sD_z(B)\cong k[x, y]/(f_z, g_z)\in \cB_{2, 2}^\vee$. 
\end{proof}

\begin{theorem} \label{thm.CF} 
$\cA_{2, 2}^\vee = \cF_2$.  
\end{theorem} 

\begin{proof} If $E\in \cG_2$, then there exists $A\in \cC_{3, 1}$ such that $E\cong C(A)\cong \sD(A^!)\in \cB_{2, 2}^\vee$ by Theorem \ref{thm.CGb}, Lemma \ref{lem.3} and Theorem \ref{thm.1}.  
On the other hand, if $E=k[x, y]/(f, g)\in \cB_{2, 2}^\vee$, then $\sH^z(f, g)\in \cB_{3, 2}$ by Theorem \ref{thm.2}, so $\sH^z(f, g)^!\in \cC_{3, 1}$ by Theorem \ref{thm.HMd}, hence
$E\cong \sD_z(\sH^z(f, g))\cong C(\sH^z(f, g)^!)\in \cG_2$ by Theorem \ref{thm.2}, Lemma \ref{lem.3} and Theorem \ref{thm.CGb}.  So far, we have proved  $\cB_{2, 2}^\vee=\cG_2$.  

It is known that $M_2(k) \cong k_{-1}[x,y]/(x^2+1,y^2+1)$, and it is easy to check that the algebra homomorphism
$$
\varphi: k (\xymatrix{1 \ar@<0.8ex>[r]^-x & 2 \ar@<0.7ex>[l]^-y}) \; /(xy, yx) \to  k_{-1} [x,y] /(x^2, y^2-1) \cong k_{-1} [x,y] /(x^2, y^2+1)
$$
defined by
$$
e_1 \mapsto \frac{1}{2}(1+y), \, e_2 \mapsto \frac{1}{2}(1-y), \, x \mapsto x+yx, \, y \mapsto x-yx
$$
is an isomorphism. By comparing the list of algebras in $\cF_2\setminus \cG_2$ 
in Theorem \ref{thm-4dimfro} with the list of 
sequences in $\cA_{2, 2}^{\vee, st}\setminus \cB_{2, 2}^{\vee, st}$in Corollary \ref{cor.ANC}, we have a surjection 
$$(\cF_2\setminus \cG_2)\to (\cA_{2, 2}^{\vee, st}\setminus \cB_{2, 2}^{\vee, st}) \to  (\cA_{2, 2}^{\vee}\setminus \cB_{2, 2}^{\vee}),$$
which is obviously injective. 
\end{proof}

The following corollary plays an essential role in the subsequent paper \cite{HMW}.  

\begin{corollary} \label{cor.nst} 
For $(f, g, h), (f', g', h')\in \cA_{2, 2}^{\vee, st}\setminus \cB_{2, 2}^{\vee, st}$, $(f, g, h)\sim_{st}(f', g', h')$ in $k\<x, y\>$ if and only if $k\<x, y\>/(f, g, h)\cong k\<x, y\>/(f', g', h')$. 
\end{corollary}

The above corollary is false for  $(f, g, h), (f', g', h')\in \cB_{2, 2}^{\vee, st}$ as the following example shows. 

\begin{example} \label{ex.stc}
$F = (x^2-1, y^2-1)$ and $F' = (x^2-y, y^2-1)$ are strongly regular normal sequences in $k[x,y]$ 
such that $k[x,y]/I_F \cong k^4 \cong k[x,y]/I_{F'}\in \cB_{2,2}^\vee$.
For every $\phi\in \Aut^{\ZZ}k\<x, y\>= \GL_2(k)$, $\phi(xy-yx)=(\det \phi)(xy-yx)$.  For every $f=\a (x^2-1)+\b (y^2-1)$ where $\a, \b\in k$, $f_1=0$, so $\phi(f)_1=0$.  It follows that $\phi(f)\neq x^2-y$, so $(x^2-1, y^2-1, xy-yx)\not \sim _{st}(x^2-y, y^2-1, xy-yx)$ in $k\<x, y\>$ (see Remark \ref{rem.stc}).  
\end{example} 
 
By Theorem \ref{thm.CF}, we can write every algebra in Theorem \ref{thm-4dimfro} in the form of $\cA_{2, 2}^\vee$. 

\begin{corollary} \label{cor-avee22}
Every 4-dimensional Frobenius algebra is isomorphic to exactly one of the algebras listed in Table \ref{tab-avee22} where $k_\l[x,y]/(x^2,y^2)\cong k_{\l'}[x, y]/(x^2, y^2)$ if and only if $\l'=\l^{\pm 1}$.
\begin{center}
\begin{table}[h]
\begin{threeparttable}
\centering
\caption{Algebras in $\cA^\vee_{2,2}$.} \label{tab-avee22}
\begin{tabular}{|c|}
\hline
{\rm $\cB_{2,2}^\vee$}  \\ \hline 
$k[x,y]/(x^2-1,y^2-1)$, $k[x,y]/(x^2-y-1,y^2-1)$, $k[x,y]/(x^2-\frac{\sqrt{3}}{2}y-1,y^2-\frac{\sqrt{3}}{2}x-1)$, \\
$k[x,y]/(x^2,y^2-1)$, $k[x,y]/(x^2,y^2-x)$, $k[x,y]/(x^2,y^2) $ \\ 
\hline \hline
{\rm $\cA_{2,2}^\vee\setminus\cB_{2,2}^\vee$}  \\ \hline  
$k_{-1} [x,y]/(x^2 + 1,y^2 + 1)$, $k_{-1} [x,y]/(x^2,y^2 + 1)$, 
$k_{-1}[x,y] / (x^2, y^2+yx)$, $k_\l[x,y]/(x^2,y^2)$ where $\l \neq 0, 1$ \\\hline
\end{tabular}
\end{threeparttable}
\end{table}
\end{center}
\end{corollary}

\end{document}